\documentclass[letterpaper,11pt,reqno]{amsart}
\usepackage{graphicx,amsmath,amssymb,amsthm, float}

\usepackage{srcltx}
\usepackage[latin1]{inputenc}
\usepackage[T1]{fontenc}

\usepackage[left=1in,right=1in, top=1in, bottom=1in]{geometry}

\newtheorem{X}{X}[section]
\newtheorem{corollary}[X]{Corollary}

\newtheorem{lemma}[X]{Lemma}

\newtheorem{proposition}[X]{Proposition}
\newtheorem{theorem}[X]{Theorem}

\newcommand{\od}{\text{ odd}}
\theoremstyle{definition}
\newtheorem{remark}[X]{Remark}

\newcommand{\eps}{\varepsilon}

\newcommand{\Sp}{\mathrm{Sp}}
\newcommand{\Sodd}{S^*_{\text{odd}}}
\newcommand{\Sodde}{S^*_{\text{\emph{odd}}}}
\newcommand{\Sodda}{S_{\text{odd}}}
\newcommand{\Sevena}{S_{\text{even}}}
\newcommand{\Soddea}{S_{\text{\emph{odd}}}}

\newcommand{\wn}{w \left( \frac dX \right)}
\newcommand{\R}{\mathbb R}
\newcommand{\D}{\mathcal D}
\newcommand{\Dstar}{\mathcal D^*(\phi;X)}
\newcommand{\Da}{\mathcal D(\phi;X)}

\renewcommand{\Re}{\operatorname{Re}}

\newcommand{\sumta}{\underset{d\neq 0}{\sum \nolimits^{*}} \widetilde w\left( \frac dX \right) }

\newcommand{\sumn}{\underset{d \text{ odd}}{\sum \nolimits^{*}} w\left( \frac dX \right) }

\newcommand{\sumtstar}{\underset{d \text{ odd}}{\sum \nolimits^{*}} w\left( \frac dX \right) }
\newcommand{\sumtstarextra}{\underset{d \text{\emph{ odd}}}{\sum \nolimits^{*}} w\left( \frac dX \right) }

\newcommand{\supp}{\text{supp}\hspace{1pt}\widehat \phi}
\newcommand{\esupp}{\emph{\text{sup}}(\emph{\text{supp}}\hspace{1pt}\widehat \phi)}
\newcommand{\suppp}{\emph{\text{supp}}\hspace{1pt}\widehat \phi}

\numberwithin{equation}{section}

\title[Low-lying zeros of quadratic Dirichlet $L$-functions]{Low-lying zeros of quadratic Dirichlet $L$-functions: \\ lower order terms for extended support}

\author{Daniel Fiorilli, James Parks and Anders S\"odergren}
\address{D\'epartement de math\'ematiques et de statistique, Universit\'e d'Ottawa, \newline
\rule[0ex]{0ex}{0ex}\hspace{8pt} 585 King Edward, Ottawa, Ontario, K1N 6N5, Canada}
\email{daniel.fiorilli@uottawa.ca}
\address{Department of Mathematics and Computer Science, University of Lethbridge,\newline
	\rule[0ex]{0ex}{0ex}\hspace{8pt}  4401 University Drive, Lethbridge, AB, T1K 3M4, Canada\newline
	\rule[0ex]{0ex}{0ex}\hspace{8pt} \textit{Present address}: Institut f\"ur Algebra, Zahlentheorie und Diskrete Mathematik, \newline
	\rule[0ex]{0ex}{0ex}\hspace{8pt} Leibniz Universit\"at Hannover, Welfengarten 1, 30167 Hannover, Germany}
\email{parks@math.uni-hannover.de}
\address{Department of Mathematical Sciences, University of Copenhagen, Universitetsparken 5,\newline
\rule[0ex]{0ex}{0ex}\hspace{8pt} 2100 Copenhagen \O, Denmark}
\email{sodergren@math.ku.dk}

\date{\today}
\begin{document}

\maketitle

\begin{abstract}

We study the $1$-level density of low-lying zeros of Dirichlet $L$-functions attached to real primitive characters of conductor at most $X$. Under the Generalized Riemann Hypothesis, we give an asymptotic expansion of this quantity in descending powers of $\log X$, which is valid when the support of the Fourier transform of the corresponding even test function $\phi$ is contained in $(-2,2)$. We uncover a phase transition when the supremum $\sigma$ of the support of $\widehat \phi$ reaches $1$, both in the main term and in the lower order terms. A new lower order term appearing at $\sigma=1$ involves the quantity $\widehat \phi (1)$, and is analogous to a lower order term which was isolated by Rudnick in the function field case.
\end{abstract}

\section{Introduction}

The study of statistics of zeros of $L$-functions was initiated in Montgomery's seminal paper \cite{Mo} on the pair-correlation of zeros of $\zeta(s)$. This work inspired \"Ozl\"uk and Snyder \cite{OS1,OS2} to prove related results on the $1$-level density of low-lying zeros of Dirichlet $L$-functions attached to real characters 
$$\chi_d(n) := \left( \frac{d}{n} \right),$$
with $d\neq 0$.\footnote{Similar results were obtained independently in an unpublished preprint of Katz and Sarnak \cite{KS1}.} These low-lying zeros of Dirichlet $L$-functions are of particular interest since they have strong connections with important problems such as the size of class numbers of imaginary quadratic number fields and Chebyshev's bias for primes in arithmetic progressions. The aforementioned results were extended to the $n$-level density for general $n$ by Rubinstein \cite{Rub}, and for extended support under GRH by Gao \cite{Gt,G}. Note that Gao considered the family 
$$ \mathcal F^* (X):= \big\{ L(s,\chi_{8d}) : 1\leq |d| \leq X;  d \text{ is odd and squarefree} \big\}, $$
which is known to have significant technical advantages over that of all real characters (see also \cite{So}). For several years it was not known how to match Gao's asymptotic with the random matrix theory predictions. However, this was recently established for $n\leq 7$ by Levinson and Miller \cite{LM}, and for all $n$ by Entin, Roditty-Gershon and Rudnick \cite{ERR}. In addition, we mention that the Ratios Conjecture of Conrey, Farmer and Zirnbauer \cite{CFZ} has been shown by Conrey and Snaith \cite{CS} to predict a precise expression for the $1$-level density; this prediction was confirmed up to a power saving error term by Miller \cite{Mi} for a restricted class of test functions. 

In this paper we study the low-lying zeros of real Dirichlet $L$-functions in the family $\mathcal F^*(X)$. Our focus will be on lower order terms in the $1$-level density, a statistic for low-lying zeros that we now introduce in detail. Throughout, $\phi$ will denote a real and even Schwartz test function. Given a (large) positive number $X$, the $1$-level density for the single $L$-function $L(s,\chi_{d})$ is the sum
\begin{align*}
D_X(\chi_{d};\phi):=\sum_{\gamma_{d}} \phi\left( \gamma_{d}\frac{  L}{2\pi} \right),
\end{align*}
with $\gamma_{d}:= -i(\rho_{d}-\frac12)$, where $\rho_{d}$ runs over the nontrivial zeros of $L(s,\chi_{d})$ (i.e. zeros with $0<\Re(\rho_d) <1$). Moreover, we set 
\begin{align}\label{Ldefinition}
L:=\log\left(\frac{X}{2\pi e}\right).
\end{align}
We consider a cutoff function $w(t)$, which is an even, nonzero and nonnegative Schwartz function. The corresponding total weight is given by
\begin{align*}
W^*(X):=\sumtstar. 
\end{align*}
Here and throughout, a star on a sum will denote a restriction to squarefree integers. We then define the $1$-level density of the family $\mathcal F^*(X)$ as the sum 
\begin{equation} \label{equation definition one level densities}
\Dstar:= \frac 1{W^*(X)}\sumtstar D_X(\chi_{8d};\phi).
\end{equation}
Our main theorem is an asymptotic expansion of this quantity in descending powers of $\log X$, which is valid when $\text{supp}\hspace{1pt}\widehat \phi \subset (-2,2)$. This is a refinement of the results of \"Ozl\"uk-Snyder \cite{OS2} and Katz-Sarnak \cite{KS1}.

\begin{theorem}
\label{theorem main 8d}
Fix $K\in \mathbb N$, assume GRH and suppose that $\suppp \subset (-2,2)$. Then the $1$-level density of low-lying zeros in the family $\mathcal F^*$ of quadratic Dirichlet $L$-functions whose conductor is an odd squarefree multiple of $8$ is given by
$$\Dstar =  \widehat \phi(0)-\frac{1}2\int_{-1}^{1} \widehat \phi(u)\,du +\sum_{k=1}^K \frac {R_{w,k}(\phi)}{(\log X)^k} +O_{w,\phi,K}\left( \frac 1{(\log X)^{K+1}}\right), $$
where the coefficients $R_{w,k}(\phi)$ are linear functionals in $\phi$ that can be given explicitly in terms of $w$ and the derivatives of $\widehat\phi$ at the points $0$ and $1$. The first coefficient is given by
\begin{multline}
R_{w,1}(\phi)= \widehat \phi(0) \bigg( \log \Big(\frac{16}{ e^{\gamma+1}} \Big)    -2\sum_{p\geq 3} \frac{\log p}{p(p^2-1)}  - 2\int_{2}^{\infty} \frac{\theta(t)-t }{t^2}dt +   \frac 2{\widehat w(0)}\int_0^{\infty} w(x) (\log x)\,dx  \bigg)  \\+2\widehat \phi (1) \int_{1}^{\infty} \Big(  H_u h_1 ( u)  + \frac{[u]}uh_2( u)\Big)\,du,
\label{equation definition of first coefficient}
\end{multline}
where $\theta(t):=\sum_{p\leq t}\log p$ is the Chebyshev function, $H_u:=\sum_{n\leq u} n^{-1}$ is the $u$-th harmonic number and $h_1,h_2$ are explicit transforms of the weight function $w$ that are defined in Section \ref{section extended support}.
\end{theorem}

Theorem \ref{theorem main 8d} will follow from the more precise Theorem \ref{theorem main 8d precise}, which gives an expression for $\Dstar$ with a power saving error term. 

We remark that Theorem \ref{theorem main 8d} agrees with the Katz-Sarnak prediction \cite{KS2, KS3}, which states that
$$ \lim_{X \rightarrow \infty} \mathcal D^* (\phi;X) = \widehat\phi(0)-\frac 12 \int_{-1}^1 \widehat \phi(u)\,du $$
independently of the support of $\widehat \phi$. This asymptotic, which was already obtained by \"Ozl\"uk-Snyder when $\supp\subset (-2,2)$ (under GRH), shows that there is a phase transition when the supremum of $\supp$ approaches $1$. Such a transition is also present in the lower order terms in Theorem \ref{theorem main 8d}, because of the terms involving $\widehat \phi^{(m)}(1)$. 

The Katz-Sarnak prediction originates from the following function field analogue of the family $\mathcal F^*(X)$. Consider the family $\mathcal H_{n,q}$ of zeta functions of hyperelliptic curves $y^2=Q(x)$ defined over $\mathbb F_q$, where $Q(x)$ is a monic square-free polynomial of degree $n$. Note the relations $n=2g+2$ if $n$ is even and $n=2g+1$ if $n$ is odd, where $g$ is the genus of the hyperelliptic curve. Using the fact that the monodromy corresponding to the family $\mathcal H_{n,q}$ equals the symplectic group $\Sp(2g)$ and an equidistribution theorem of Deligne, Katz and Sarnak proved precise results for the low-lying zeros of the zeta functions in $\mathcal H_{n,q}$ in the limit as both $q$ and $n$ tend to infinity (see \cite{KS2,KS3}).

The family $\mathcal H_{n,q}$ with $q$ fixed and $n=2g+1$ was also studied by Rudnick \cite{Rud}. He considered the associated $1$-level density in the limit as $n\rightarrow \infty$. Note that this limit is expected to be a more direct analogue to number fields than the $q\rightarrow \infty $ limit. Restricting to the case when $\supp\subset(-2,2)$, Rudnick gave the following estimate for the $1$-level density of low-lying zeros of the zeta functions in $\mathcal H_{n,q}$:
\begin{equation}\label{Rudnicks result}
\widehat\phi(0)-\frac 12\int_{-1}^1 \widehat \phi(u)du +\frac 1g \bigg( \widehat \phi(0) \bigg( \sum_{P \text{ monic irr.}} \frac{\text{deg }P}{q^{2\text{deg }P}-1}+\frac{1}2 \bigg)- \widehat \phi(1)\frac{q+1}{2(q-1)} \bigg)+o\left( \frac 1g \right)
\end{equation}
(cf.\ \cite[Corollary 3 and the subsequent paragraph]{Rud}; see also \cite{BCDGL,Ch}). Recall that when translating between function fields and number fields  it is customary to set $g=\log X$. Taking this into account, note the striking similarity between the expression in \eqref{Rudnicks result} and our Theorem \ref{theorem main 8d}; in particular, they both contain a lower order term involving $\widehat \phi(0)$ and $\widehat \phi(1)$. Here it is interesting that the prediction from the function field situation indicates not only the main term in the number field case (as in the Katz-Sarnak philosophy), but also lower order terms. 

In this connection, we note that a lower order term involving $\widehat \phi(1)$ is also present in the $1$-level density of the family of Dirichlet $L$-functions attached to all characters modulo $q$ (see \cite[Theorem 1.2]{FM}). However, in this family this term is of order $X^{-\frac 12}/\log X$ and is thus much smaller than in the family $\mathcal F^*$ of real characters.

Next we study the family of all real characters $\chi_d$ ordered by the modulus $|d|$, that is, we consider
$$ \mathcal F (X):= \big\{ L(s,\chi_{d}) : 1\leq |d| \leq X \big\}. $$
Note that $\zeta(s) \in \mathcal F (X)$, and that for any $a\in \mathbb N$, the functions $L(s,\chi_{d})$ and $L(s,\chi_{a^2d})$ have the same nontrivial zeros. The reason why we allow such repetitions is that it simplifies the analysis and allows one to obtain significantly sharper error terms\footnote{There are several known examples in the literature of families with repetitions having such advantages (cf.\ \cite{Y2,FM,FPS,SST}).} (compare the error terms in Theorems \ref{theorem main 8d precise} and \ref{theorem all d}). 

Similarly as above, we define the $1$-level density of the family $\mathcal F(X)$ to be the sum 
\begin{equation}\label{definition one level density all d}
\Da:= \frac 1{W(X)} \sum_{d \neq 0} \wn D_X(\chi_{d};\phi), 
\end{equation}
where
\begin{align*}
W(X):=\sum_{d\neq 0}\wn.
\end{align*}
Our second main theorem is an asymptotic formula for $\Da$ valid when $\text{supp}\hspace{1pt}\widehat \phi \subset (-2,2)$. For convenience we introduce the notation $\sigma:=\text{sup}(\supp)$.

\begin{theorem}
\label{theorem all d}
Fix $\epsilon>0$. Assume GRH and suppose that $\suppp \subset (-2,2)$. Let
\begin{align*}
U_1(X):=\frac 1{2\sqrt{2\pi e}\mathcal Mw(1)} \bigg( \mathcal Mw(\tfrac 12) - \frac{w(0)}{\sqrt X}\bigg)\int_{-1}^1 \left(\frac X{2\pi e}\right)^{\frac{u-1}2} \widehat \phi(u)\,du
\end{align*}
and 
\begin{multline*}
U_2(X):=\frac{\tfrac 12 \mathcal Mw(\tfrac 12)}{\sqrt{2\pi e}\mathcal Mw(1)} \int_0^1 \left(\frac X{2\pi e}\right)^{\frac{u-1}2} \widehat \phi(u)\,du\\-\frac 2{L\widehat w(0)}\int_{0}^{\infty} \bigg( \widehat \phi( 1+\tfrac {\tau}L  )   e^{\frac{\tau}2} \sum_{n \geq 1} \widehat h \big(  n e^{\frac{\tau}2}\big)  +\widehat \phi( 1-\tfrac {\tau}L  )  \sum_{n\geq 1} h\big(  n e^{\frac{\tau}2}\big)\bigg)\,d\tau,
\end{multline*}
where $\mathcal M w$ denotes the Mellin transform of $w$ and $h(t):=\widehat w(2\pi e t^2)$. Let further 
\begin{equation*}
U(X):=\begin{cases}
U_1(X) & \text{ if } \sigma < 1, \\
U_2(X) & \text{ if } 1 \leq \sigma < 2.
\end{cases}
\end{equation*}
Then the $1$-level density of low-lying zeros in the family $\mathcal F$ of all quadratic Dirichlet $L$-functions is given by
\begin{align} \Da &= \widehat \phi(0)+\int_{1}^{\infty} \widehat \phi(u)\,du +\frac {\widehat \phi(0)}{L} \left( \log\left(\frac{e^{1-\gamma}}{2^{\frac23}}\right) + \frac 2{\widehat w(0)}\int_0^{\infty} w(x) (\log x)\,dx + 2\frac{\zeta'(2)}{\zeta(2)} -\frac{\mathcal Mw(\tfrac 12)}{\mathcal M w(1)}\frac{\zeta(\tfrac12)}{X^{\frac 12}}\right) \nonumber\\
 & \hspace{.5cm}-\frac 2{L   }\sum_{j=1}^{\infty}\sum_{p} \frac{\log p}{p^{j}} \left(  1+\frac 1p\right)^{-1} \widehat \phi\left( \frac{2j \log p}{L} \right) \nonumber
+\frac{1}{L}\int_0^\infty\frac{e^{-x/2}+e^{-3x/2}}{1-e^{-2x}}\left(\widehat{\phi}(0)-\widehat{\phi}\left(\frac{x}{L}\right)\right) dx\\
& \hspace{1cm}+U(X)+O_{w,\phi,\eps}\big(X^{\eta(\sigma)+\eps}\big), \label{equation thm all d} 
\end{align}
where 
\begin{align*}
\eta(\sigma):=\begin{cases}
-\frac35 & \text{ if } \sigma < 1, \\
\frac{\sigma}2-1 & \text{ if } 1 \leq \sigma < 2.
\end{cases}
\end{align*}
\end{theorem}

The term $U(X)$ in Theorem \ref{theorem all d} is $O(X^{\frac{\sigma-1}{2}})$ when $\text{supp}\hspace{1pt}\widehat \phi \subset (-1,1)$, but is of order $(\log X)^{-1}$ when $\widehat \phi$ has mass in a neighborhood of $1$ (see Lemma \ref{lemma expansion of U(X)}). Therefore, this term is responsible for
 a phase transition at $1$. Moreover, Lemma \ref{lemma Seven} shows that 
$$ - \frac 2{L   }\sum_{j=1}^{\infty}\sum_{p}\frac{\log p}{p^{j}} \left(  1+\frac 1p\right)^{-1} \widehat \phi\left( \frac{2j \log p}{L} \right) = -\frac {\phi (0)}2 + O_{\phi}\left( \frac1{\log X}\right),$$
and combining this with Lemma \ref{lemma expansion of U(X)} and taking $X\rightarrow \infty$ in \eqref{equation thm all d}, we recover the Katz-Sarnak prediction.

We now briefly describe the tools used in the proofs of Theorems \ref{theorem main 8d} and \ref{theorem all d}. The fundamental tool is an application of Poisson summation to the prime sum in the explicit formula, following the work of Katz and Sarnak \cite{KS1}. In contrast with our previous work \cite{FPS}, terms with square index in the resulting sum are now of considerable size and contribute to both the main term and new lower order terms in the $1$-level density. The novelty in the present work is to transform the terms of square index with an additional application of Poisson summation which isolates the Katz-Sarnak main term and other tractable terms that are estimated later in the analysis (see Lemma \ref{lemma I_s(X)}).

Finally, we observe that for small support we have an even more precise result for the family $\mathcal F$ (which is in fact unconditional). It is interesting to note that the error term we obtain in this case is significantly sharper than the error term predicted by the corresponding Ratios Conjecture (cf.\ \cite[Theorem 3.1]{CS}).

\begin{theorem}\label{theorem all d small support}
Fix $\epsilon>0$ and suppose that $\suppp \subset (-1,1)$. Then the $1$-level density of low-lying zeros in the family $\mathcal F$ of all quadratic Dirichlet $L$-functions is given by
\begin{align} \Da &= \frac {\widehat \phi(0)}{LW(X)}  \sumta\log|d|-\frac {\widehat \phi(0)}{L} \log\left(\pi e^{\gamma}2^{\frac{5}{3}}\right) \nonumber\\
 & -\frac 2{L   }\sum_{j=1}^{\infty}\sum_{p} \frac{\log p}{p^{j}} \left(  1+\frac 1p\right)^{-1} \widehat \phi\left( \frac{2j \log p}{L} \right) \nonumber
+\frac{1}{L}\int_0^\infty\frac{e^{-x/2}+e^{-3x/2}}{1-e^{-2x}}\left(\widehat{\phi}(0)-\widehat{\phi}\left(\frac{x}{L}\right)\right) dx\\
& +\frac 1{2\sqrt{2\pi e}\mathcal Mw(1)} \bigg( \mathcal Mw(\tfrac 12)   -  \frac{w(0)}{\sqrt X}\bigg)\int_{-1}^1 \left(\frac X{2\pi e}\right)^{\frac{u-1}2} \widehat \phi(u)\,du
 +O_{w,\phi,\eps}\big(X^{ \xi(\sigma)+\eps}\big),
\end{align}
where $\widetilde w (y):= \sum_{n\geq 1} w(n^2y)$ and
$$\xi(\sigma):= \begin{cases}
-1+\sigma & \text{ if } \frac 1{2m+1} \leq \sigma < \frac 1{2m+\frac 12}, \\
-\frac{4m-1}{4m+1} & \text{ if } \frac 1{2m+\frac 12} \leq \sigma < \frac 1{2m-1},
\end{cases} $$
for each $m\geq 1$. 
\end{theorem}

\begin{remark}
We can also obtain an unconditional result similar to Theorem \ref{theorem all d small support} for $\Dstar$. However, in this case the error term will be the weaker  $O_{\eps}(X^{\frac{\sigma-1}2+\eps})$ and hence we choose not to provide the details in the present paper. (Note that under GRH, Proposition \ref{proposition GRH small support} gives the sharper error term $O_{\eps}\big(X^{\max\{\frac{\sigma}4-\frac12,\frac{3\sigma}4-\frac34\}+\eps}\big)$). This setting was previously studied by Miller \cite[Theorem 1.2]{Mi}, who claimed an error term of size 
$$O_{\eps}\big(X^{-\frac12}+X^{-(1-\frac32\sigma)+\eps}+X^{-\frac34(1-\sigma)+\eps}\big).$$ 
However, going through the proof of \cite[Lemma 3.5]{Mi}, we find\footnote{The second, third and fourth terms on the right-hand side of \cite[(3.39)]{Mi} should be multiplied by $X$ (see \cite[(3.36)]{Mi}).} that the actual error term resulting from \cite{Mi} is
 $O_{\eps}(X^{\frac{\sigma-1}2+\eps}).$ 
\end{remark}

\section*{Acknowledgments}

We are grateful to Ze\'ev Rudnick for pointing our attention to his paper \cite{Rud} and for his questions which motivated the present work. A substantial part of this work was accomplished while the authors were visiting the University of Copenhagen, the Banff International Research Station and the Centre de Recherches Math\'ematiques. 
We are very thankful to these institutions for providing us with excellent research conditions. The first author was supported at the University of Ottawa by an NSERC discovery grant and at the Institut Math\'ematique de Jussieu by a postdoctoral fellowship from the Fondation Sciences Math\'ematiques de Paris. The second author was supported at the University of Lethbridge by a PIMS Postdoctoral Fellowship. The third author was supported by a grant from the Danish Council for Independent Research
and FP7 Marie Curie Actions-COFUND (grant id: DFF-1325-00058).

\section{Preliminary results for the family $\mathcal F^*(X)$}

\subsection{Explicit formula and character sums}

We will study the $1$-level density via the explicit formula for primitive Dirichlet $L$-functions.

\begin{lemma}[Explicit Formula]
\label{lemma explicit formula}
Assume that $\phi$ is an even Schwartz test function whose Fourier transform has compact support. Then the $1$-level density defined in \eqref{equation definition one level densities} is given by the formula
\begin{align} 
\Dstar =  &\frac {\widehat \phi(0)}{ LW^*(X)}\sumtstarextra\log |d| - \frac {\widehat\phi(0)}{L} (\gamma+\log \pi) \nonumber\\
-&\frac 2{L  W^*(X)  }\sum_{p,m} \frac{\log p}{p^{m/2}} \widehat \phi\left( \frac{m \log p}{L} \right) \sumtstarextra \left( \frac{8d}{p^m}\right) \nonumber\\ +&\frac{1}{L}\int_0^\infty\frac{e^{-x/2}+e^{-3x/2}}{1-e^{-2x}}\left(\widehat{\phi}(0)-\widehat{\phi}\left(\frac{x}{L}\right)\right) dx.
\label{equation one level density explicit formula squarefree}
\end{align}
\end{lemma}

\begin{proof}
Let $d$ be an odd squarefree integer. Then $\chi_{8d}$ is a primitive character of conductor $8|d|$. Taking $\widehat{F}(t)=\Phi\left(\frac12+it\right):=\phi\left(\frac{tL}{2\pi}\right)$ in  \cite[Theorem 12.13]{MV} (whose conditions are satisfied by our restrictions on $\phi$), we obtain the formula
\begin{multline*} D_X(\chi_{8d};\phi)= \frac{\widehat \phi(0)}L \left(\log \bigg(\frac{8|d|}{\pi}\bigg)+ \frac{\Gamma'}{\Gamma}\left(\frac{1}{4}+\frac{\mathfrak a}{2}\right)\right)-\frac 2L\sum_{p,m} \frac{\chi_{8d} (p^m) \log p}{p^{m/2}} \widehat \phi\left( \frac{m \log p}{L} \right) \\
 +\frac 2L \int_0^\infty\frac{e^{-(1/2+\mathfrak a)x}}{1-e^{-2x}}\left(\widehat{\phi}(0)-\widehat{\phi}\left(\frac{x}{L}\right)\right) dx,
\end{multline*}
where
$$ \mathfrak a := \begin{cases}
0& \text{ if } d\geq 0,\\
1 & \text{ if } d < 0.
\end{cases}$$
 Formula \eqref{equation one level density explicit formula squarefree} then follows by summing over $d$ against the weight function $w$. 
\end{proof}

We will need the following estimate on a weighted quadratic character sum.

\begin{lemma}
\label{lemma count of squarefree}
Fix $n\in \mathbb N$ and $\varepsilon>0$. Under the Riemann Hypothesis (RH), we have the estimate
\begin{equation*}
\sumtstarextra \left(\frac {8d}n\right)= \kappa(n) \frac{2X}{3\zeta(2)} \widehat w(0)   \prod_{p\mid n}  \left(1+\frac 1{p} \right)^{-1}+O_{\varepsilon,w}\big(|n|^{\frac 38(1-\kappa(n))+\varepsilon}X^{\frac 14+\varepsilon}\big),
\end{equation*}
where
$$ \kappa(n) := \begin{cases}
1 &\text{ if } n \text{ is an odd square}, \\
0 &\text{ otherwise.}
\end{cases}$$
\end{lemma}

\begin{proof}
The result follows similarly as in \cite[Lemma 2.10]{FPS}.
\end{proof}

\begin{remark}
\label{remark total weight *}
Taking $n=1$ in Lemma \ref{lemma count of squarefree} gives the following conditional estimate for the total weight:
\begin{equation}
W^*(X)= \frac{2X}{3\zeta(2)} \widehat w(0)   +O_{\varepsilon,w}\big(X^{\frac 14+\varepsilon}\big).
\end{equation}
\end{remark}

Let us now evaluate the first sum on the right-hand side of \eqref{equation one level density explicit formula squarefree}.

\begin{lemma} \label{lemma logd}
Fix $\varepsilon>0$, and assume the Riemann Hypothesis (RH). We have the estimate
\begin{equation*} 
\frac 1{W^*(X)} \sumtstarextra \log |d|= \log X + \frac 2{\widehat w(0)}\int_0^{\infty} w(x) (\log x) \,dx   +O_{\varepsilon,w}\big(X^{-3/4 +\varepsilon}\big).  
\end{equation*}
\end{lemma}

\begin{proof}
The proof is similar to that of \cite[Lemma 2.8]{FPS}.
\end{proof}

The following consequence of GRH will be central in our analysis.

\begin{lemma}
\label{lemma Riemann bound}
Assume GRH. For $m\in \mathbb Z_{\neq 0}$ and $y\geq 1$ we have the estimate
$$S_{m}(y):= \sum_{p\leq y} \left( \frac m p\right) \log p  = \delta_{m=\square}y+O\big(y^{\frac 12} \log(2y)\log(2|m|y)\big). $$
\end{lemma}

\begin{proof}
Write $m=a^2b$, with $\mu^2(b)=1$. Then we clearly have that $S_m(y) = S_b(y)+ O(\log |a|)$. Applying \cite[Thm.\ 5.15]{IK}, we have that 
$$ \sum_{\substack{p^e\leq y \\ e\geq 1}} \log p \left( \frac{b}{p^e}\right) = \delta_{b=1}y+ O\big(y^{\frac 12} \log(2y)\log(2|b|y)\big).  $$
The result follows by trivially bounding the contribution of prime powers.
\end{proof}

\subsection{Poisson summation}\label{section Poisson 8d}

In this section we will provide an approximate expression for the prime sum appearing in \eqref{equation one level density explicit formula squarefree}. We first separate the odd and the even prime powers, by writing
\begin{equation}
 S^*_{\text{odd}}:= -\frac 2{L  W^*(X)  }\sum_{\substack{p \\ m \text{ odd}}} \frac{ \log p}{p^{m/2}} \widehat \phi\left( \frac{m \log p}{L} \right) \sumn \left( \frac{8d}{p^m}\right),
 \label{equation Sodd first appearance}
\end{equation} 
and similarly for $S^*_{\text{even}}.$ We will transform \eqref{equation Sodd first appearance} using Poisson summation.

From now on, we will not necessarily indicate the dependence of the error terms on $\phi$ and $w$.

\begin{lemma}
Fix $\varepsilon>0$. Assume GRH, and suppose that $\sigma=\esupp< \infty $. Then, for any $S\geq 1$, we have the estimate
\begin{multline}\label{equation thing to prove in lemma Poisson }
S^*_{\text{\emph{odd}}}  =  - \frac {2X}{LW^*(X)}  \sum_{\substack{ s \leq S \\ s \text{\emph{ odd}}}} \frac{\mu(s)}{s^2} \sum_{\substack{p \nmid 2 s }} \frac{ \overline{\epsilon_p} \log p}{p}\widehat \phi\left( \frac{\log p}{L} \right)   \sum_{t \in \mathbb Z}  \bigg( \left( \frac {-2t}{p} \right)\widehat w\left( \frac{Xt}{s^2p}\right)- \frac 12\left( \frac {-t}{p} \right)\widehat w\left( \frac {Xt}{2s^2p} \right)\bigg)\\+O_{\varepsilon}\big(  X^{-\frac 34+\varepsilon}+X^{\varepsilon}(\log S)^3S^{-1}\big),
\end{multline}
where
$$\epsilon_p:= \begin{cases}
1 &\text{ if } p\equiv 1 \bmod 4, \\
i &\text{ if } p\equiv 3 \bmod 4.
\end{cases} $$
\label{lemma:Poisson}
\end{lemma}

\begin{proof}
By Lemma \ref{lemma count of squarefree}, the contribution of the terms with $m\geq 3$ in \eqref{equation Sodd first appearance} is $O_{\eps}(X^{-\frac 34+\eps})$.
We transform the sum over $d$ into a sum over all odd integers using the usual convolution identity for the indicator function of squarefree integers. This yields the estimate 
\begin{equation*}
S^*_{\text{odd}} = -\frac 2{L  W^*(X)  }\sum_{\substack{  s \in \mathbb N \\ s \od}} \mu(s)\sum_{ \substack{u \in \mathbb Z \\ u \text{ odd}}}\sum_{p\nmid 2s } \frac{\log p}{p^{1/2}} \widehat \phi\left( \frac{ \log p}{L} \right)  w\left( \frac{us^2}{X}\right)\left( \frac{8u}{p}\right) +O_{\varepsilon}\big(X^{-\frac 34+\varepsilon}\big). 
\end{equation*}
We now apply Lemma \ref{lemma Riemann bound} and summation by parts. Note that if $u$ is odd, then $8u$ is never a square. It follows that
the terms with $s> S$ are 
$$ \ll \frac 1{LX} \sum_{\substack{  s >S \\ s \od}} \sum_{\substack{u \in \mathbb Z \\ u \text{ odd}}} w\left( \frac{us^2}{X}\right) \big(\log(2|u|sX)\big)^3 \ll_{\varepsilon} X^{\varepsilon}(\log S)^3S^{-1}.$$ 
As for the terms with $s\leq S$, we introduce additive characters using Gauss sums, resulting in the estimate
\begin{multline*}
S^*_{\text{odd}} = -\frac 2{L  W^*(X)  }\sum_{\substack{  s \leq S \\ s \od}} \mu(s)\sum_{p\nmid 2s } \frac{\overline{\epsilon_p}\log p}{p} \widehat \phi\left( \frac{ \log p}{L} \right) \sum_{b\bmod p} \left( \frac {b}p\right)  \\
\times \sum_{u \in \mathbb Z} \bigg( w\left(\frac{us^2 }X \right)  e\left(\frac{8u b}p \right) - w\left(\frac{2us^2 }X \right)  e\left(\frac{16u b}p \right)\bigg) +O_{\varepsilon}\big(X^{-\frac 34+\varepsilon}+X^{\varepsilon}(\log S)^3S^{-1}\big). 
\end{multline*} 
Applying Poisson summation in the inner sum yields the expression
\begin{multline*}
S^*_{\text{odd}} = -\frac{2X}{LW^*(X)}  \sum_{\substack{  s \leq S \\ s \od}} \frac{\mu(s)}{s^2} \sum_{\substack{ p\nmid 2s}} \frac{\overline{\epsilon_p}  \log p}{p}\widehat \phi\left( \frac{ \log p}{L} \right) \sum_{b\bmod p} \left( \frac bp\right) \\  
\times  \bigg(  \sum_{v_1 \in \mathbb Z}\widehat w\left(  \frac X{s^2}\left( v_1- \frac {8b}p\right) \right)- \frac 12\sum_{v_2 \in \mathbb Z}\widehat w\left(  \frac {X}{2s^2}\left( v_2- \frac {16b}p\right) \right) \bigg)
+O_{\varepsilon}\big(X^{-\frac 34+\varepsilon}+X^{\varepsilon}(\log S)^3S^{-1}\big).
\end{multline*}
The sums over $b$ and $v_j$ can be replaced by a single sum over $t_j:= v_jp-2^{j+2}b$ ($j=1,2$). Using the fact that for $p>2$, we have
\begin{align*}
\left(\frac{b}{p}\right)=\left(\frac{2}{p}\right)\left(\frac{8b}{p}\right)=\left(\frac{-2(t_1-v_1p)}{p}\right)=\left(\frac{-2t_1}{p}\right) 
\end{align*}
and
\begin{align*}
\left(\frac{b}{p}\right)=\left(\frac{16b}{p}\right)=\left(\frac{v_2p-t_2}{p}\right)=\left(\frac{-t_2}{p}\right),
\end{align*} 
we end up with the estimate \eqref{equation thing to prove in lemma Poisson }.
\end{proof}

\begin{lemma}\label{lemma:small s}
Assume GRH, fix $\varepsilon>0$ and suppose that $\sigma=\esupp< \infty $. Then, for any $1 \leq S \leq X^{2}$, we have that\footnote{This range can be replaced by $1 \leq S \leq X^{M}$, for any fixed $M\in \mathbb N$. However, the important range for our analysis is $1 \leq S \leq X^2$.}
\begin{multline*} 
S^*_{\text{\emph{odd}}} = \frac{2X}{W^*(X)} \sum_{\substack{  s \leq S \\ s \text{\emph{ odd}}}} \frac{\mu(s)}{s^2}  \int_{0}^{\infty} \widehat \phi( u )  \sum_{m\geq 1} \bigg( \frac 12 \widehat w\left( \frac{m^2 X^{1-u}}{2s^2(2\pi e)^{-u}}\right)-\widehat w\left( \frac{2m^2 X^{1-u}}{s^2(2\pi e)^{-u}}\right) \bigg)\,du\\
+O_{\varepsilon}\big(X^{-\frac 34+\varepsilon}+X^{\varepsilon}S^{-1}+SX^{\frac{\sigma}2-1+\varepsilon}\big).  
\end{multline*}
\end{lemma}

\begin{proof}
By the definition of $\epsilon_p$, the second part of the main term in \eqref{equation thing to prove in lemma Poisson } equals
\begin{align*}
&\frac {X}{LW^*(X)} \sum_{\substack{ s \leq S \\ s \od}} \frac{\mu(s)}{s^2} \sum_{\substack{p \nmid 2 s }} \frac{ \log p}{p}\widehat \phi\left( \frac{\log p}{L} \right)   \sum_{t \in \mathbb Z}\left( \frac{1+i}2\left( \frac {t}{p} \right) +\frac{1-i}2 \left( \frac {-t}{p} \right)  \right) \widehat w\left( \frac{Xt}{2s^2p}\right) \\
&=\frac {X}{LW^*(X)} \sum_{\substack{ s \leq S \\ s \od}} \frac{\mu(s)}{s^2} \sum_{\substack{p \nmid 2 s }} \frac{ \log p}{p}\widehat \phi\left( \frac{\log p}{L} \right)   \sum_{t >0}\left( \left( \frac {t}{p} \right) + \left( \frac {-t}{p} \right)  \right) \widehat w\left( \frac{Xt}{2s^2p}\right).
\end{align*}
Note that we can add back the primes dividing $2s$ at the cost of an admissible error term.

By Lemma \ref{lemma Riemann bound}, we have for $t>0$ and $y\geq 1$ that 
$$  T_t(y) := \sum_{p\leq y} \log p \left( \left( \frac t{p} \right) + \left( \frac {-t}{p} \right) \right) = \delta_{t=\square}(y-1) +O\big(y^{\frac 12} \log(2y)\log (2|t|y)\big).$$
It then follows that
\begin{align*}
&\sum_{t > 0}\sum_{\substack{p  }} \frac{ \log p}{p}\widehat \phi\left( \frac{\log p}{L} \right)  \left( \left( \frac t{p} \right) + \left( \frac {-t}{p} \right)  \right) \widehat w\left( \frac{Xt}{2s^2p}\right)=\sum_{t > 0}\int_{1}^{\infty}  \widehat \phi\left( \frac{\log y}{L} \right) \widehat w\left( \frac{Xt}{2s^2y}\right) \frac{dT_t(y)}y \\
&= -\sum_{t > 0}\int_{1}^{\infty}  \bigg[ y^{-1} \widehat \phi\left( \frac{\log y}{L} \right)\widehat w\left( \frac{Xt}{2s^2y}\right)\bigg]' \big(\delta_{t=\square}(y-1) +O\big(y^{\frac 12} \log(2y)\log (2|t|y)\big) \big)\,dy  \\
&= \sum_{\substack{ t = \square }} \int_{1}^{\infty}  \widehat \phi\left( \frac{\log y}{L} \right)\widehat w\left( \frac{Xt}{2s^2y}\right) \frac{dy}y +O_{\varepsilon}\big(s^2(\log (2s))^2X^{\frac{\sigma}2-1+\varepsilon}\big),
\end{align*} 
by an argument similar to that in the proof of \cite[Lemma 4.3]{FPS}\footnote{Note that the integrand in the current paper is zero for $y\geq (X/2\pi e)^{\sigma}$.}. As for the first part of the main term of \eqref{equation thing to prove in lemma Poisson }, it can  be analyzed along the same lines; the quantity analogous to $T_t(y)$ is
$$  \sum_{p\leq y} \log p \left( \left( \frac {2t}{p} \right) + \left( \frac {-2t}{p} \right) \right) = \delta_{t=2\square}(y-1)+ O\big(y^{\frac 12} \log(2y)\log (2|t|y)\big).$$ 
The lemma follows from taking the change of variables $u=\log y/L$ and summing over $s$.
\end{proof}

\section{New lower order terms}

In Lemma \ref{lemma:small s} we saw that to understand $S^*_{\text{odd}}$, it is important to give a precise estimate of the term
\begin{equation}\label{equation definition I_S(X)}
I_s(X):=  \int_{0}^{\infty} \widehat \phi( u )  \sum_{m\geq 1} \widehat w\left( \frac{2m^2 X^{1-u}}{s^2(2\pi e)^{-u}}\right) du. 
\end{equation} 
Indeed, the lemma implies that for $S\leq X^2$ and under GRH,
\begin{equation}\label{equation Sodd in terms of I_s(X)} 
\Sodd = \frac{2X}{W^*(X)} \sum_{\substack{  s \leq S \\ s \od}} \frac{\mu(s)}{s^2}  \big(\tfrac 12I_{2s}(X)-I_s(X) \big)+O_{\varepsilon}\big(X^{-\frac 34+\varepsilon}+X^{\varepsilon}S^{-1}+SX^{\frac{\sigma}2-1+\varepsilon}\big). 
\end{equation}
Here and throughout this section, we assume that $\sigma=\text{sup}(\supp)< \infty$. Our strategy will be to treat the integrals over the intervals $[0,1]$ and $[1,\sigma]$ differently; the former will be computed directly and the latter via an application of Poisson summation.

\subsection{Small support}

In this section we assume that $\sigma<1$. In this range we will not find new lower order terms; these only appear when $\sigma$ is at least $1$ (see Section \ref{section extended support}). 

\begin{proposition}
\label{proposition GRH small support}
Fix $\epsilon>0$. Assume GRH and suppose that $\sigma<1$. Then we have the bound
$$  \Sodde \ll_{\varepsilon} X^{\frac{\sigma}4 - \frac 12+\varepsilon} + X^{\frac{3\sigma}4-\frac 34+\varepsilon}. $$
\end{proposition}

\begin{proof}
Let 
$$ T(t):= \sum_{\substack{s\leq t \\ s \od}} \frac{\mu(s)}{s^2} = \frac 4{3\zeta(2)} +O_{\eps} \big(t^{-\frac 32 + \eps}\big). $$
We then have, for $0\leq u \leq 1$, that
\begin{align*}
\sum_{\substack{  s \leq S \\ s \od}} \frac{\mu(s)}{s^2} &\widehat w\left( \frac{2m^2 X^{1-u}}{s^2(2\pi e)^{-u}}\right)  = \int_{0^+}^S \widehat w\left( \frac{2m^2 X^{1-u}}{t^2(2\pi e)^{-u}}\right) dT(t) \\
&= \widehat w\left( \frac{2m^2 X^{1-u}}{S^2(2\pi e)^{-u}}\right) T(S) + \frac{4m^2X^{1-u}}{(2\pi e)^{-u}}\int_{0^+}^S  \widehat w' \left( \frac{2m^2 X^{1-u}}{t^2(2\pi e)^{-u}}\right) \left( \frac 4{3\zeta(2)} +O_{\eps} \big(t^{-\frac 32 + \eps}\big)\right) \frac{dt}{t^3} \\
&\ll_{\eps} S^{-\frac 32 +\eps} \bigg|\widehat w\left( \frac{2m^2 X^{1-u}}{S^2(2\pi e)^{-u}}\right)\bigg|+m^2X^{1-u}\int_{0^+}^S  \bigg|\widehat w' \left( \frac{2m^2 X^{1-u}}{t^2(2\pi e)^{-u}}\right) \bigg| \frac{dt}{t^{\frac 92-\eps}}.
\end{align*}
Note that the part of the last integral for $ t \in (0,X^{\frac{1-u}2-\eps}]$ is $O_{N,\eps}\big((mX)^{-N}\big)$ for any $N\geq 1$, by the rapid decay of $\widehat w'$. Summing over $m$ and integrating over $u$, we obtain that
\begin{align*}
\sum_{\substack{  s \leq S \\ s \od}} \frac{\mu(s)}{s^2} I_s(X) & \ll_{\eps} \int_{0}^{\infty} \big|\widehat \phi( u )\big|  \sum_{m\geq 1}  \Bigg(S^{-\frac 32 +\eps}\bigg|\widehat w\left( \frac{2m^2 X^{1-u}}{S^2(2\pi e)^{-u}}\right)\bigg|\\ 
&\hspace{2cm}+m^2X^{1-u}\int_{X^{\frac{1-u}2-\eps}}^S  \bigg|\widehat w' \left( \frac{2m^2 X^{1-u}}{t^2(2\pi e)^{-u}}\right) \bigg| \frac{dt}{t^{\frac 92-\eps}} \Bigg)\,du +X^{-1} \\
& \ll \int_{0}^{\infty} \big|\widehat \phi( u )\big|  \bigg(S^{-\frac 32 +\eps} \frac{S}{X^{\frac{1-u}2}}+\frac 1{X^{\frac{1-u}2}}\int_{X^{\frac{1-u}2-\eps}}^S   \frac{dt}{t^{\frac 32-\eps}} \bigg)\,du + X^{-1}\\
& \ll \frac{X^{\frac{\sigma-1}2}}{S^{\frac 12 -\eps}} +  X^{\frac{3\sigma}4 -\frac 3 4+\varepsilon}.
\end{align*}
Hence, from \eqref{equation Sodd in terms of I_s(X)} it follows that for $S\leq X^2$,
$$ \Sodd \ll_{\eps}\frac{X^{\frac{\sigma-1}2}}{S^{\frac 12 -\eps}}+X^{\frac{3\sigma}4 -\frac 3 4+\varepsilon} +X^{\varepsilon}S^{-1}+SX^{\frac{\sigma}2-1+\varepsilon}.$$
The result follows by taking $S=X^{\frac{1}2-\frac{\sigma}{4}}$.
\end{proof}

\subsection{Extended support}\label{section extended support}

In this section we will see that when $\sigma>1$ the prime sum $\Sodd$ contains terms of considerable size, and we will give an asymptotic expansion of these terms in descending powers of $\log X$. For convenience, we introduce the function
\begin{equation}\label{equation definition g}
g(y):=\widehat w(4\pi e y^2).
\end{equation}

\begin{lemma}\label{lemma I_s(X)}
Suppose that $\sigma=\esupp<\infty$. Then, for $s\geq 1$, the quantity defined in \eqref{equation definition I_S(X)} satisfies the estimate
\begin{multline}\label{equation estimate I_s(X)}
I_s(X)=\frac 1L\int_{0}^{\infty} \bigg( \widehat \phi( 1+\tfrac {\tau}L  )   se^{\frac{\tau}2} \sum_{n \geq 1} \widehat g \big(  sn e^{\frac{\tau}2}\big)  +\widehat \phi( 1-\tfrac {\tau}L  )  \sum_{m\geq 1} g\bigg( \frac{  m e^{\frac{\tau}2}}{s}\bigg)\bigg)\,d\tau  \\ 
+\frac{s\widehat g(0)}2 \int_{1}^{\infty} \left( \frac X{2\pi e}\right)^{\frac{u-1}2} \widehat\phi( u )\,du-\frac{\widehat w (0)}2\int_{1}^{\infty} \widehat \phi(u)\,du   +O\big(sX^{-\frac 12}\big).
\end{multline}
\end{lemma}

\begin{proof}
Extending the integral in \eqref{equation definition I_S(X)} to $\R$ and making the substitution $\tau = L(u-1)$, we obtain
$$ I_s(X)=\frac 1L\int_{-\infty}^{\infty} \widehat \phi( 1+\tfrac {\tau}L )  \sum_{m\geq 1} \widehat w\left( \frac{4\pi  m^2 e^{1-\tau}}{s^2}\right)d\tau +O\big(sX^{-\frac 12}\big).$$
We denote the integrals over $(-\infty,0]$ and $[0,\infty)$ by $I_s^-(X)$ and $I_s^+(X)$, respectively.  For the second of these integrals we apply Poisson summation. We obtain
\begin{align*}
I_s^+(X)&= \frac 1L\int_{0}^{\infty} \widehat \phi( 1+\tfrac {\tau}L  )  \bigg(-\frac{\widehat w (0)}2  + \frac 12\sum_{m\in \mathbb Z}g \bigg( \frac{ m e^{-\frac{\tau}2}}{s}\bigg)\bigg)\,d\tau \\
&=  \frac 1L\int_{0}^{\infty} \widehat \phi( 1+\tfrac {\tau}L )  \bigg(-\frac{\widehat w (0)}2  + \frac{se^{\frac{\tau}2}}2 \sum_{n\in \mathbb Z} \widehat g \big(  sn e^{\frac{\tau}2}\big)\bigg)\,d\tau \\
&= \frac 1L\int_{0}^{\infty} \widehat \phi( 1+\tfrac {\tau}L  )  \bigg(\frac{se^{\frac{\tau}2}\widehat g(0)}2-\frac{\widehat w (0)}2   + se^{\frac{\tau}2} \sum_{n \geq 1} \widehat g \big(  sn e^{\frac{\tau}2}\big)\bigg)\,d\tau.
\end{align*}
For $I_s^-(X)$, we substitute $\tau$ with $-\tau$, which gives
$$I^-_s(X)=\frac 1L\int_{0}^{\infty} \widehat \phi( 1-\tfrac {\tau}L  )  \sum_{m\geq 1} g\bigg( \frac{  m e^{\frac{\tau}2}}{s}\bigg)\,d\tau. $$
The lemma follows by combining the above formulas for  $I_s^-(X)$ and $I_s^+(X)$.
\end{proof}

We define the functions
$$ h_1(x) := \frac{3\zeta(2)}{\widehat w(0)}\sum_{\substack{ s\geq 1 \\ s \text{ odd}}} \frac {\mu(s)}s \big(\widehat g(2sx)-\widehat g(sx)\big);\hspace{1cm} h_2(x):=\frac{3\zeta(2)}{\widehat w(0)}\sum_{\substack{ s\geq 1 \\ s \text{ odd}}} \frac {\mu(s)}{s^2} \big( \tfrac 12g( \tfrac x{2s})-g(\tfrac xs)\big).  $$
It is a routine exercise to check that $h_1(x)$ and $h_2(x)$ are smooth for $x\in \mathbb R_{\neq 0}$ and $x\in \mathbb R$, respectively. One can also check that for any fixed $N\geq 1$ and $\eps >0$, we have the bounds $h_1(x)\ll_N x^{-N}$ and (under RH) $h_2(x)\ll_{\eps} x^{-\frac 32+\eps}$.

\begin{remark}
One can show that $h_1$ is continuous at $0$. Indeed, let $f(u)=\widehat g(2u)-\widehat g(u)$ and write, for
$x\neq0$,
\begin{equation}\label{equation remark continuity}
\sum_{\substack{ s\geq 1 \\ s \text{ odd}}} \frac {\mu(s)}s f(sx) = \int_{1^-}^{\infty}f(tx)\,dS(t) = -\int_{1^-}^{\infty} xf'(tx)S(t)\,dt,
\end{equation}
where 
$$S(t)=\sum_{\substack{s\leq t \\ s \text{ odd}}} \frac{\mu(s)}s \ll t^{-\frac 12+\epsilon}$$ 
(under RH). We then apply the (rough) bound $f'(u) \ll |u|^{-3/4}$ ($u\neq 0$) and conclude that the right-hand side of \eqref{equation remark continuity} is $\ll |x|^{\frac 14}$, proving continuity. 
\end{remark}

\begin{corollary}
Fix $\epsilon>0$ and assume GRH. Then we have the estimate
$$ \Sodde =   \int_{1}^{\infty} \widehat \phi(u)\,du+ \frac 1L\int_{0}^{\infty} \bigg( \widehat \phi( 1+\tfrac {\tau}L  )   e^{\frac{\tau}2} \sum_{n \geq 1}  h_1\big(n e^{\frac{\tau}2}\big)  +\widehat \phi( 1-\tfrac {\tau}L  )  \sum_{n\geq 1} h_2\big(ne^{\frac{\tau}2}\big)\bigg)\,d\tau+O_{\eps}\big(X^{\frac{\sigma}6-\frac 13+\eps}\big).$$
\label{corollary Sodd in terms of J}
\end{corollary}

\begin{proof}
We sum the right-hand side of \eqref{equation estimate I_s(X)} over $s$. By \eqref{equation Sodd in terms of I_s(X)} and Remark \ref{remark total weight *}, in the range $S\leq X^2$ this gives the estimate
\begin{multline} 
\Sodd = \frac{2X}{LW^*(X)} \sum_{\substack{  s \leq S \\ s \od}} \frac{\mu(s)}{s^2}  \Bigg( \frac 12 \int_{0}^{\infty} \bigg[ \widehat \phi( 1+\tfrac {\tau}L  )   2se^{\frac{\tau}2} \sum_{n \geq 1} \widehat g \big(  2sn e^{\frac{\tau}2}\big)  +\widehat \phi( 1-\tfrac {\tau}L  )  \sum_{m\geq 1} g\bigg( \frac{  m e^{\frac{\tau}2}}{2s}\bigg)\bigg]\,d\tau\\
-\int_{0}^{\infty} \bigg[ \widehat \phi( 1+\tfrac {\tau}L  )   se^{\frac{\tau}2} \sum_{n \geq 1} \widehat g \big(  sn e^{\frac{\tau}2}\big)  +\widehat \phi( 1-\tfrac {\tau}L  )  \sum_{m\geq 1} g\bigg( \frac{  m e^{\frac{\tau}2}}{s}\bigg)\bigg]\,d\tau \Bigg)  \\
+\int_{1}^{\infty} \widehat \phi(u)\,du +O_{\varepsilon}\big( X^{-\frac 12+\varepsilon}+X^{\varepsilon}S^{-1}+SX^{\frac{\sigma}2-1+\varepsilon}\big).
\end{multline}
We can extend the sum over $s$ to all positive odd integers at the cost of the error term $O\big(X^{\eps}S^{-\frac 12}\big)$. Changing order of summation, taking $S=X^{\frac23 - \frac{\sigma}3}$ and applying Remark \ref{remark total weight *} gives the result.
\end{proof}

We summarize the findings of this section in the following theorem.

\begin{theorem}\label{theorem main 8d precise}
Fix $\epsilon>0$. Assume GRH and suppose that $\sigma=\esupp < 2$. Then the $1$-level density of low-lying zeros in the family $\mathcal F^*$ of quadratic Dirichlet $L$-functions whose conductor is an odd squarefree multiple of $8$ is given by
\begin{align}\label{Theorem 3.5} 
\Dstar =\,\,& \widehat \phi(0)+\int_{1}^{\infty} \widehat \phi(u)\,du+ \frac{\widehat \phi(0)}L \bigg( \log (2 e^{1-\gamma}) + \frac 2{\widehat w(0)}\int_0^{\infty} w(x) (\log x)\,dx \bigg)   \nonumber\\
&+\frac{1}{L}\int_0^\infty\frac{e^{-x/2}+e^{-3x/2}}{1-e^{-2x}}\left(\widehat{\phi}(0)-\widehat{\phi}\left(\frac{x}{L}\right)\right) dx\\
&- \frac 2{L   }\sum_{\substack{p>2 \\ j\geq 1}} \frac{\log p}{p^{j}} \left(  1+\frac 1p\right)^{-1} \widehat \phi\left( \frac{2j \log p}{L} \right)  +J(X)+O_{\eps}\big(X^{\frac{\sigma}6-\frac13+\eps}\big), \notag
\end{align}
where 
$$ J(X):= \frac 1L\int_{0}^{\infty} \bigg( \widehat \phi( 1+\tfrac {\tau}L  )   e^{\frac{\tau}2} \sum_{n \geq 1}  h_1\big(n e^{\frac{\tau}2}\big)  +\widehat \phi( 1-\tfrac {\tau}L  )  \sum_{n\geq 1} h_2\big( n e^{\frac{\tau}2}\big)\bigg)\,d\tau.$$
\end{theorem}

\begin{proof}
Combining Lemma \ref{lemma explicit formula} with Lemma \ref{lemma logd} and Corollary  \ref{corollary Sodd in terms of J}, and 
noting that Lemma \ref{lemma count of squarefree} implies the estimate
\begin{equation}\label{equation Seven estimate}
S^*_{\text{even}} = - \frac 2{L   }\sum_{\substack{p>2 \\ j\geq 1}} \frac{\log p}{p^{j}} \left(  1+\frac 1p\right)^{-1} \widehat \phi\left( \frac{2j \log p}{L} \right)+O_{\eps}\big(X^{-\frac 34+\eps}\big), 
\end{equation}
we obtain the desired result.
\end{proof}

Next we show how to deduce Theorem \ref{theorem main 8d} from this result. The key is to expand the various terms in the right-hand side of \eqref{Theorem 3.5} in descending powers of $\log X$. Note that the term $J(X)$ is of order $(\log X)^{-1}$ and constitutes a genuine lower order term in the $1$-level density $\Dstar$. 

\begin{lemma}\label{lemma expansion}
Assume RH. Then, for any $K\geq 1$, we have the expansion
$$ J(X)=\sum_{k=1}^K \frac{c_{w,k} \widehat \phi^{(k-1)}(1)}{L^k}+ O_{K}\left( \frac 1{L^{K+1}} \right), $$
where the constants $c_{w,k}$ can be given explicitly. The first of these constants is given by
$$c_{w,1} =  2\int_{1}^{\infty} \Big(H_u h_1 ( u)  + \frac{[u]}uh_2( u)\Big)\,du,$$
where $H_u:=\sum_{n\leq u} n^{-1}$ is the $u$-th harmonic number.
\end{lemma}

\begin{proof}
By the decay properties of $h_1$ and $h_2$, we have that
$$ J(X)= \frac 1L\int_{0}^{L^{\frac 12}} \bigg( \widehat \phi( 1+\tfrac {\tau}L)e^{\frac{\tau}2} \sum_{n \geq 1}  h_1 \big(n e^{\frac{\tau}2}\big)  +\widehat \phi( 1-\tfrac {\tau}L  )  \sum_{n\geq 1} h_2\big(n e^{\frac{\tau}2}\big)\bigg)\,d\tau+O_{\eps}\big(\exp\big(-(\tfrac34-\eps)\sqrt {L}\big)\big).$$
We can now expand $\widehat \phi$ in Taylor series, resulting in the expression
\begin{align*}
J(X)&= \sum_{k=1}^K \frac {\widehat \phi^{(k-1)}(1)}{(k-1)!L^k} \int_{0}^{L^{\frac 12}} \bigg(  \tau^{k-1} e^{\frac{\tau}2} \sum_{n \geq 1}  h_1\big(n e^{\frac{\tau}2}\big)  + (-\tau)^{k-1} \sum_{n\geq 1} h_2\big( n e^{\frac{\tau}2}\big)\bigg)\,d\tau+O_K\big(L^{-K-1}\big)\\
&= \sum_{k=1}^K \frac {\widehat \phi^{(k-1)}(1)}{(k-1)!L^k} \sum_{n\geq 1} \int_{0}^{\infty} \Big(  \tau^{k-1} e^{\frac{\tau}2}   h_1\big(n e^{\frac{\tau}2}\big)  + (-\tau)^{k-1} h_2\big(ne^{\frac{\tau}2}\big)\Big)\,d\tau+O_K\big(L^{-K-1}\big).
\end{align*} 
Finally, the first summand equals
\begin{align*}
\frac{ \widehat \phi (1)}L \sum_{n\geq 1} \int_{0}^{\infty} \Big(   e^{\frac{\tau}2}   h_1\big(n e^{\frac{\tau}2}\big)  +  h_2\big( ne^{\frac{\tau}2}\big)\Big)\,d\tau &= \frac{ 2\widehat \phi (1)}L \sum_{n\geq 1} \int_{n}^{\infty} \bigg(  \frac{    h_1 ( u)}n  + \frac{h_2( u)}u\bigg)\,du.
\end{align*} 
The result follows from interchanging the order of summation and integration.
\end{proof} 

The final ingredient needed in the proof of Theorem \ref{theorem main 8d} is an expansion for $S^*_{\text{even}}$ of the same form as that of $J(X)$ in Lemma \ref{lemma expansion}.

\begin{lemma}\label{lemma Seven}
Suppose that $\sigma=\esupp<\infty$. Then we have the formula
$$-\frac 2{L}\sum_{\substack{p>2 \\ j\geq 1 }} \frac{\log p}{p^j} \left(  1+\frac 1p\right)^{-1} \widehat \phi\left( \frac{2 j \log p}{L} \right) =-\frac{\phi(0)}2 + \sum_{k=1}^K \frac{d_k \widehat \phi^{(k-1)}(0)}{L^k} +O_{K}\left( \frac 1{L^{K+1}} \right), $$
where the coefficients $d_k$ are real numbers that can be given explicitly. In particular we have
$$ d_1=  -2\sum_{p,j\geq 3} \frac{\log p}{p^j} \left(1+\frac 1p \right) ^{-1} -2+3\log 2-2 \int_{2}^{\infty} \frac{\theta(t)-t }{t^2}  dt. $$
\end{lemma} 

\begin{proof}
Let $\delta= 1/(K+2)$ and set $\xi = (\log X)^{-1+\delta}$. The sum of the terms with $j\geq 2$ equals
\begin{align*}
-\frac 2{L} &\sum_{\substack{p^j \leq X^{\xi} \\ p>2 \\ j\geq 2 }} \frac{\log p}{p^{j}} \left(  1+\frac 1p\right)^{-1} \Bigg(\sum_{k=0}^K \frac{\widehat \phi^{(k)}(0)}{k!} \left(\frac{2j \log p}L \right)^k+O_K\bigg( \bigg(\frac{ \log p^{2j}}{L}\bigg)^{K+1} \bigg) \Bigg)+O\big(e^{-\frac12(\log X)^{\delta}}\big) \\
&= - \frac {2 }{L} \sum_{k=0}^K \frac{\widehat \phi^{(k)}(0)}{k!L^k}\sum_{\substack{p^j \leq X^{\xi} \\ p>2 \\j\geq 2}} \frac{\log p(2j \log p)^k}{p^{j}} \left(  1+\frac 1p\right)^{-1} +O_K\big(L^{-K-2}\big) \\
&=  - \frac {2 }{L} \sum_{k=0}^K \frac{\widehat \phi^{(k)}(0)}{k!L^k}\sum_{\substack{p>2 \\ j\geq 2}} \frac{\log p(2j \log p)^k}{p^{j}} \left(  1+\frac 1p\right)^{-1} +O_K\big(L^{-K-2}\big),
\end{align*}
which is of the desired form.

As for the terms with $j=1$, we observe that the sum over $p\leq X^{\xi}$ is given by
\begin{align*} 
-\frac 2{L}\sum_{\substack{ 2< p \leq  X^{\xi}  }} \frac{\log p}{p} &\left(  1+\frac 1p\right)^{-1} \Bigg(\sum_{k=0}^K \frac{\widehat \phi^{(k)}(0)}{k!} \left(\frac{2 \log p}L \right)^k+O_{K}\bigg( \bigg(\frac{ \log p}{L}\bigg)^{K+1} \bigg) \Bigg) \\
& =- \sum_{k=0}^K \frac{\widehat \phi^{(k)}(0)}{k!L^{k+1}}\sum_{\substack{2<p \leq X^{\xi}  }} \frac{(2\log p)^{k+1}}{p} \left(  1+\frac 1p\right)^{-1} +O_{K}\big(L^{(-1+\delta)(K+2)}\big)\\
& =-\sum_{k=0}^K \frac{\widehat \phi^{(k)}(0)}{k!L^{k+1}} \bigg( \frac{(2\xi\log X)^{k+1}}{k+1}+ \ell_k+ \sum_{n\geq 1} (-1)^n  \sum_{\substack{p>2 }} \frac{(2\log p)^{k+1}}{p^{n+1}} \bigg)  +O_{K}\big(L^{-K-1}\big),
\end{align*}
where 
$$\ell_k := (2\log 2)^k \Big( 2-\frac{k+3}{k+1} \log 2 \Big)- 2^{k+1}\int_{2}^{\infty} \frac{(\theta(t)-t) (\log t)^{k-1} ( k-\log t)}{t^2}\,dt. $$
The terms with $p>X^{\xi}$ are handled by writing
\begin{align*}
-\frac 2{L}\sum_{\substack{p> X^{\xi}}} \frac{\log p}{p} &\left(  1+\frac 1p\right)^{-1} \widehat \phi\left( \frac{2 \log p}{L} \right)= -\frac 2{L}\int_{X^{\xi}} ^{\infty}  \frac{1}{t} \widehat \phi\left( \frac{2 \log t}{L} \right) d\theta(t) +O(X^{-\xi})\\
& = \frac 2{L}\widehat \phi\left( \frac{2 \xi \log X}{L} \right)
+\frac 2{L}\int_{X^{\xi}} ^{\infty} \bigg[ \frac{1}{t}\widehat \phi\left( \frac{2 \log t}{L} \right) \bigg]' \big(t+O(t e^{-c\sqrt{\log t}})\big)\,dt + O\big(e^{-c(\log X)^{\frac {\delta}2}}\big)\\ 
&= \frac 2{L}\widehat \phi\left( \frac{2 \xi \log X}{L} \right)
+\frac 2{L}\int_{X^{\xi}} ^{\infty} \bigg[ \frac{1}{t}  \widehat \phi\left( \frac{2 \log t}{L} \right) \bigg]' t\,dt + O\big(e^{-\frac c2(\log X)^{\frac {\delta}2}}\big)\\
&= -\frac 2{L}\int_{X^{\xi}} ^{\infty}  \frac{1}{t}  \widehat \phi\left( \frac{2 \log t}{L} \right) dt + O\big(e^{-\frac c2(\log X)^{\frac {\delta}2}}\big)\\
&=-\int_{2\xi \log X/L} ^{\infty}   \widehat \phi\left(u\right)\,du + O\big(e^{-\frac c2(\log X)^{\frac {\delta}2}}\big)\\
&=-\frac 12 \phi(0)+\sum_{k=0}^{K} \frac{\widehat \phi^{(k)}(0)}{(k+1)!} \left( \frac{2\xi\log X}L \right)^{k+1} +O_{K}\big(L^{(-1+\delta)(K+2)}\big).
\end{align*}
The result follows from combining the above estimates. Note in particular that the terms involving $\xi$ cancel.
\end{proof}

We are now ready to prove our main theorem.

\begin{proof}[Proof of Theorem \ref{theorem main 8d}]
Combining Theorem \ref{theorem main 8d precise} and Lemmas \ref{lemma expansion} and \ref{lemma Seven}, we obtain the formula
\begin{align} 
&\Dstar = \widehat \phi(0)-\frac{\phi(0)}2+\int_{1}^{\infty} \widehat \phi(u)\,du+ \frac{\widehat \phi(0)}L \bigg( \log (2 e^{1- \gamma}) + \frac 2{\widehat w(0)}\int_0^{\infty} w(x) (\log x)\,dx \bigg)   \nonumber\\
&\hspace{0cm}+\frac{1}{L}\int_0^\infty\frac{e^{-x/2}+e^{-3x/2} }{1-e^{-2x}}\left(\widehat{\phi}(0)-\widehat{\phi}\left(\frac{x}{L}\right)\right)dx  +\sum_{k=1}^K \frac {d_{k}\widehat \phi^{(k-1)}(0)}{L^k} +\sum_{k=1}^K \frac{c_{w,k} \widehat \phi^{(k-1)}(1)}{L^k} +O_K\left( \frac 1{L^{K+1}} \right). \notag
\end{align}
Writing the gamma factor as
$$\frac{1}{L}\int_0^\infty\frac{e^{-x/2}+e^{-3x/2} }{1-e^{-2x}}\left(\widehat{\phi}(0)-\widehat{\phi}\left(\frac{x}{L}\right)\right)dx= -\sum_{j=1}^{K-1} \frac{\widehat \phi^{(j)}(0)}{j! L^{j+1}} \int_0^{\infty} \frac{x^j(e^{-x/2}+e^{-3x/2})}{1-e^{-2x}}\,dx+O_K\left( \frac 1{L^{K+1}} \right),$$
the desired result clearly follows.
\end{proof}

\section{The family $\mathcal F(X)$}

\subsection{Preliminaries}
For convenience, we define the even smooth function $\widetilde w: \mathbb R \setminus \{0\} \rightarrow \mathbb R $ by
\begin{equation} \label{eq:wtildedef}
\widetilde w (x):= \sum_{n\geq 1} w(n^2x).
\end{equation}
It follows that $\widetilde w(x)$ decays rapidly as $x\rightarrow \infty$, and that its Mellin transform satisfies
$\mathcal M \widetilde w (s) = \zeta(2s) \mathcal M w(s)$ (see \cite[Lemma 2.3]{FPS}). Moreover,
note that $\widetilde w(x)$ blows up near $x=0$. Applying the explicit formula, we now give an expression for $\mathcal D(\phi;X)$.

\begin{lemma}[Explicit Formula]\label{lemma explicit formula all d}
Assume that $\phi$ is an even Schwartz test function whose Fourier transform has compact support. Then the $1$-level density defined in \eqref{definition one level density all d} is given by the formula
\begin{align} 
\Da = &\frac {\widehat \phi(0)}{LW(X)}  \sumta\log \bigg(\frac{|d|}{\pi}\bigg)- \frac {\widehat\phi(0)}{L} \bigg(\gamma+\log 4 \bigg(1- \frac{1}{W(X)}\underset{\substack{d>0\\ \text{$d$ \emph{odd}}}}{\sum \nolimits^{*}} \widetilde w\left( \frac{2d}{X} \right)\bigg)  \bigg) \nonumber\\
-&\frac 2{L W(X)  }\sum_{p,m} \frac{\log p}{p^{m/2}}  \widehat \phi\left( \frac{m \log p}{L} \right) \sumta  \left( \frac{d}{p^m}\right) +\frac{2}{W(X)}\widetilde{w}\left(\frac1X\right)\phi\left(\frac{iL}{4\pi}\right)\nonumber\\
+&\frac{1}{L}\int_0^\infty\frac{e^{-x/2}+e^{-3x/2}}{1-e^{-2x}}\left(\widehat{\phi}(0)-\widehat{\phi}\left(\frac{x}{L}\right)\right)dx.
\label{equation one level density explicit formula all d}
\end{align}
\end{lemma}

\begin{proof}
We first note that $D_X(\chi_{dm^2};\phi)=D_X(\chi_d;\phi)$ for any $m\geq 1$. Hence, by the definition of $\widetilde w$,
\begin{align}
W(X) \mathcal D(\phi;X) = \sum_{d\neq 0} w\left( \frac dX \right) D_X(\chi_d;\phi) 
&= \sumta D_X(\chi_{d};\phi).
\end{align}
We also note that the conductor of $\chi_d$ for $d$ squarefree is given by $4^{\mathfrak b} |d|$, where
$$ \mathfrak b:= \begin{cases}
0 & \text{ if } d>0 \text{ is odd}, \\
1 & \text{ otherwise}.
\end{cases} $$
As in the proof of Lemma \ref{lemma explicit formula}, we apply \cite[Theorem 12.13]{MV} and obtain the formula
\begin{multline*} 
D_X(\chi_d;\phi)= \frac{\widehat \phi(0)}L \left(\log \bigg(\frac{4^{\mathfrak b}|d|}{\pi}\bigg)+ \frac{\Gamma'}{\Gamma}\left(\frac{1}{4}+\frac{\mathfrak{a}}{2}\right)\right)-\frac 2L\sum_{p,m} \frac{\chi_d (p^m) \log p}{p^{m/2}} \widehat \phi\left( \frac{m \log p}{L} \right) \\
+2\delta_{d=1}\phi\left(\frac{iL}{4\pi}\right)+\frac 2L \int_0^\infty\frac{e^{-(1/2+\mathfrak{a})x}}{1-e^{-2x}}\left(\widehat{\phi}(0)-\widehat{\phi}\left(\frac{x}{L}\right)\right)dx.
\end{multline*}
Formula \eqref{equation one level density explicit formula all d} then follows by summing over squarefree $d$ against the weight function $\widetilde w$. 
\end{proof}

We now give estimates on sums of the weight function $\widetilde w$. Recall that 
$$ W(X)= \sumta = \sum_{d\neq 0} \wn$$
is the total weight.

\begin{lemma}\label{lemmaweighted}
Fix $n \in \mathbb N$ and $\varepsilon>0$. We have the estimates
\begin{align*} 
W(X)&= X  \widehat w(0)  + O_{\varepsilon,w}\big(X^{\varepsilon}\big); \\
\sumta  \left( \frac dn\right) &=\kappa(n)X  \widehat w(0) \prod_{p\mid n} \left( 1+\frac 1 p\right)^{-1} 
+O_{\varepsilon,w}\big(|n|^{\frac 12-\frac{\kappa(n)}2+\varepsilon}X^{\varepsilon}\big),
\end{align*}
where
$$ \kappa(n) := \begin{cases}
1 &\text{ if } n = \square, \\
0 &\text{ otherwise.}
\end{cases}$$
Under RH, we also have 
\begin{multline*}
\frac 1{W(X)} \sumta  \log |d|= \log X + \frac 2{\widehat w(0)}\int_0^{\infty} w(x) (\log x)\,dx +2\frac{\zeta'(2)}{\zeta(2)} -\frac{\mathcal Mw(\tfrac 12)}{\mathcal M w(1)}\zeta(\tfrac12)X^{-\frac 12}\\ 
\hspace{6cm}+O_{\varepsilon,w}\big(X^{-\frac{3}{4} +\varepsilon}\big).
\end{multline*}
\end{lemma}

\begin{proof}
The result follows exactly as in \cite[Lemmas 2.5 and 2.8]{FPS}.
\end{proof}

\subsection{Poisson summation}

In this section we analyze
\begin{equation}\label{equation definition Sodd all d}
\Sodda :=  -\frac 2{L W(X)  }\sum_{\substack{p \\ m \text{ odd}}} \frac{\log p}{p^{m/2}}  \widehat \phi\left( \frac{m \log p}{L} \right) \sumta  \left( \frac{d}{p^m}\right) 
\end{equation} 
using Poisson summation (see also Section \ref{section Poisson 8d}).\footnote{$\Sevena$ is defined analogously.}

\begin{lemma}\label{lemma main term for Sodd all d}
Assume $\sigma=\esupp<1$, and let $m\in \mathbb N$ be such that $\frac 1{2m+1}\leq\sigma< \frac 1{2m-1}$. Then, for any fixed $\varepsilon >0$, we have the bound
\begin{equation*} 
\Soddea \ll _{\varepsilon} X^{-\max\{ \frac{4m-1}{4m+1},1-\sigma\}+\varepsilon} .  
\end{equation*}
Furthermore, if $1\leq\sigma< 2 $, then under GRH we have that
\begin{equation}\label{equation estimate main term for Sodd all d}  
\Soddea = -\frac{2X}{W(X)}  \int_{0}^{\infty} \widehat \phi( u )  \sum_{m\geq 1}   \widehat w\left( \frac{m^2 X^{1-u}}{(2\pi e)^{-u}}\right) du+O_{\eps}\big(X^{\frac{\sigma}2-1+\eps}\big). 
\end{equation}
\end{lemma}

\begin{proof}
We proceed as in \cite[Section 3]{FPS}. Applying the identity
$$  \sumta \left( \frac{d}{p}\right)= \sum_{k\geq 0} \sum_{d\neq 0} w\left( \frac{d}{X/p^{2k}}\right)\left( \frac{d}{p}\right)   $$
and arguing as in \cite[Lemmas 3.2 and 3.3]{FPS}, we see that 
\begin{multline*}
S_{\text{odd}}  =   -\frac {2X}{LW(X)} \sum_{0\leq k\leq  10 \log X}  \sum_{\substack{ p> X^{\frac{1-\varepsilon}{2k+1}}}} \frac{ \overline{\epsilon_p} \log p}{p^{1+2k}}\widehat \phi\left( \frac{\log p}{L} \right)    \sum_{t \in \mathbb Z}\left( \frac{-t}{p} \right) \widehat w\left( \frac{Xt}{p^{1+2k}}\right)+O_{\varepsilon}\left(  X^{-1+\varepsilon}\right),
\end{multline*}
where 
$$\epsilon_p= \begin{cases}
1 &\text{ if } p\equiv 1 \bmod 4, \\
i &\text{ if } p\equiv 3 \bmod 4.
\end{cases} $$
If $\sigma< 1 $, then the proof is similar to that of  \cite[Proposition 3.6]{FPS}. As for the case $\sigma< 2$, we argue as in \cite[Lemma 3.9]{FPS} and see that the terms with $k\geq 1$ are $O_{\eps}(X^{-\frac 23+\eps}).$ Finally, in the terms with $k=0$ we can add back the primes $p\leq X^{1-\eps}$ at the cost of a negligible error term. The resulting sum is handled in a similar way to Lemma \ref{lemma:small s}, and the estimate \eqref{equation estimate main term for Sodd all d} follows.
\end{proof}

\subsection{The new lower order terms}

In this section we treat the new lower order terms that appear in the family $\mathcal F$.

\begin{lemma}\label{lemma new poisson}
Suppose that $\sigma=\esupp<\infty$. Then we have the estimate\footnote{Note that the first term on the right-hand side equals $X^{\frac{\sigma}2-\frac 12+o(1)}$.}
\begin{multline}
\int_{0}^{\infty} \widehat \phi( u )  \sum_{m\geq 1}   \widehat w\left( \frac{m^2 X^{1-u}}{(2\pi e)^{-u}}\right) du=
\frac{\widehat h(0)}2 \int_{1}^{\infty} \left( \frac X{2\pi e}\right)^{\frac{u-1}2} \widehat\phi( u )\,du-\frac{\widehat w (0)}2\int_{1}^{\infty} \widehat \phi(u)\,du\\ 
+\frac 1L\int_{0}^{\infty} \bigg( \widehat \phi( 1+\tfrac {\tau}L  )   e^{\frac{\tau}2} \sum_{n \geq 1} \widehat h \big(n e^{\frac{\tau}2}\big)  +\widehat \phi( 1-\tfrac {\tau}L  )  \sum_{n\geq 1} h\big(n e^{\frac{\tau}2}\big)\bigg)\,d\tau+O\big(X^{-\frac12}\big),
\label{equation estimate I(X)}
\end{multline}
 where $h(y):=\widehat w(2\pi e y^2)$. 
\end{lemma}

\begin{proof}
The proof is similar to that of Lemma \ref{lemma I_s(X)}. 
\end{proof}

We now give an estimate for the fourth term on the right-hand side of \eqref{equation one level density explicit formula all d}. For $1\leq \sigma<2$, we will see that the term that arose from principal characters in \eqref{equation one level density explicit formula all d} will essentially cancel the main term in Lemma \ref{lemma new poisson}.

\begin{lemma}\label{lemma two terms cancel}
Fix $\epsilon>0$. Then, for $\sigma <1 $, we have that
\begin{equation*} 
\frac{2}{W(X)}\widetilde{w}\left(\frac1X\right)\phi\left(\frac{iL}{4\pi}\right)   = \frac 1{2\sqrt{2\pi e}\mathcal Mw(1)} \bigg( \mathcal Mw(\tfrac 12)   -  \frac{w(0)}{\sqrt X}\bigg)\int_{-1}^1 \left(\frac X{2\pi e}\right)^{\frac{u-1}2} \widehat \phi(u)\,du +O_{\eps}\big(X^{\eps-1}\big). 
\end{equation*}
As for $1\leq\sigma <2 $, we have
\begin{multline*} 
\hspace{-8pt}\frac{2}{W(X)}\widetilde{w}\left(\frac1X\right)\phi\left(\frac{iL}{4\pi}\right)-\frac{\widehat h(0)X}{W(X)} \int_{1}^{\infty} \left( \frac X{2\pi e}\right)^{\frac{u-1}2} \widehat\phi( u )\,du   = \frac{\tfrac 12 \mathcal Mw(\tfrac 12)}{\sqrt{2\pi e}\mathcal Mw(1)} \int_0^1 \left(\frac X{2\pi e}\right)^{\frac{u-1}2} \widehat \phi(u)\,du\\ +O\big(X^{\frac{\sigma}2-1}\big). 
\end{multline*}
\end{lemma}

\begin{proof}
First, an application of Poisson summation shows that, for $X\geq 1$ and arbitrary $N\geq 1$,
\begin{equation}\label{equation lemma sum of main terms in explicit formula}
\widetilde{w}\left(\frac1X\right) = \frac{X^{\frac 12}}2 \int_{\R} w(t^2)\,dt - \frac{w(0)}2 +O_N\big(X^{-N}\big).  
\end{equation}
Moreover, we have that
$$\phi\left(\frac{iL}{4\pi}\right)=\int_{\R} \left( \frac X{2\pi e}\right)^{\frac{u}2} \widehat \phi (u)\,du.$$
By trivially bounding the integral on the interval $(-\infty,0]$, it follows that for $1\leq\sigma<2$ we have
\begin{align*}
\frac{2}{W(X)}\widetilde{w}\left(\frac1X\right)\phi\left(\frac{iL}{4\pi}\right) = \frac {X\int_{\R} w(t^2)\,dt}{\sqrt{2\pi e}W(X)} \int_0^{\infty}  \left( \frac X{2\pi e}\right)^{\frac{u-1}2} \widehat \phi (u)\,du +O\big(X^{\frac{\sigma}2-1}\big).
\end{align*}
The last step is to apply the Fourier identity
$$ \int_{\R} w(t^2)\,dt = \int_{\R} \widehat w(t^2)\,dt, $$
which follows from combining Plancherel's identity with the fact that $|x|^{-\frac 12}$ is its own Fourier transform. Finally, in the case $\sigma<1$ we use a similar argument but we keep the secondary term in the expansion \eqref{equation lemma sum of main terms in explicit formula}. The result follows.
\end{proof}

\begin{proof}[Proof of Theorem \ref{theorem all d small support}]
The proof is obtained by combining Lemmas \ref{lemma explicit formula all d}, \ref{lemmaweighted}, \ref{lemma main term for Sodd all d} and \ref{lemma two terms cancel}, with the expression for $S_{\text{even}}$ analogous to \eqref{equation Seven estimate}.
\end{proof}

The rest of the section is devoted to the proof of Theorem \ref{theorem all d}.

\begin{corollary}\label{corollary formula Sodd all d}
Fix $\epsilon>0$. Assume GRH and suppose that $1\leq\sigma=\esupp < 2$. Then we have the estimate
$$ \Soddea+\frac{2}{W(X)}\widetilde{w}\left(\frac1X\right)\phi\left(\frac{iL}{4\pi}\right) =   \int_{1}^{\infty} \widehat \phi(u)\,du+U_2(X) +O_{\eps}\big(X^{\frac{\sigma}2-1+\eps}\big),
$$
where
\begin{multline*}
U_2(X):= \frac{\tfrac 12 \mathcal Mw(\tfrac 12)}{\sqrt{2\pi e}\mathcal Mw(1)} \int_0^1 \left(\frac X{2\pi e}\right)^{\frac{u-1}2} \widehat \phi(u)\,du\\-\frac 2{L\widehat w(0)}\int_{0}^{\infty} \bigg( \widehat \phi( 1+\tfrac {\tau}L  )   e^{\frac{\tau}2} \sum_{n \geq 1} \widehat h \big(  n e^{\frac{\tau}2}\big)  +\widehat \phi( 1-\tfrac {\tau}L  )  \sum_{n\geq 1} h\big(  n e^{\frac{\tau}2}\big)\bigg)\,d\tau.
\end{multline*}
\end{corollary}

\begin{proof}
The proof follows by combining Lemmas \ref{lemma main term for Sodd all d}, \ref{lemma new poisson} and \ref{lemma two terms cancel}. 
\end{proof}

In the final lemma we give an expansion of $U_2(X)$ in descending powers of $\log X$, which by Theorem \ref{theorem all d} shows that such an expansion is also possible for the one-level density $\D(\phi;X)$.

\begin{lemma}\label{lemma expansion of U(X)}
For any $K\geq 1$, we have the expansion
$$U_2(X)=\sum_{k=1}^K \frac{v_{w,k} \widehat\phi^{(k-1)}(1)}{L^k} +O_K\big(L^{-K-1}\big), $$
where the constants $v_{w,k}$ can be given explicitly. The first of these constants is given by
$$v_{w,1}=\frac{\mathcal Mw(\tfrac 12)}{\sqrt{2\pi e}\mathcal Mw(1)}- \frac 4{\widehat w(0)}\int_{1}^{\infty}\Big(H_u \widehat h( u)+\frac{[u]}uh( u)\Big)\,du, $$
where $H_u=\sum_{n\leq u} n^{-1}$ is the $u$-th harmonic number.
\end{lemma}

\begin{proof}
The second term in $U_2(X)$ can be expanded exactly as in the proof of Lemma \ref{lemma expansion}. As for the first term, we have
\begin{align*}
\int_0^1 \left(\frac X{2\pi e}\right)^{\frac{u-1}2} \widehat \phi(u)\,du&= \int_{1-L^{-\frac 12}}^1 \left(\frac X{2\pi e}\right)^{\frac{u-1}2}  \Bigg(\sum_{k=1}^K \widehat \phi^{(k-1)}(1) \frac{(u-1)^{k-1}}{(k-1)!}+O\big((u-1)^{K}\big) \bigg)\,du \\&\hspace{6cm} +O\big(e^{-\frac 12\sqrt{L}}\big)\\
&= -\sum_{k=1}^K  \frac{(-2)^k\widehat \phi^{(k-1)}(1)}{L^k}  +O_K\big(L^{-K-1}\big),
\end{align*}  
which completes the proof. 
\end{proof} 

We are now ready to prove Theorem \ref{theorem all d}.

\begin{proof}[Proof of Theorem \ref{theorem all d}]
The proof is achieved by combining Lemmas \ref{lemma explicit formula all d}, \ref{lemmaweighted}, \ref{lemma main term for Sodd all d} and \ref{lemma two terms cancel}, with the expression for $S_{\text{even}}$ analogous to \eqref{equation Seven estimate} and Corollary \ref{corollary formula Sodd all d}.
\end{proof}

\end{document}